\documentclass[preprint,3p,times]{elsarticle}

\usepackage{amssymb}
\usepackage{amsmath}
\usepackage[all,cmtip]{xy}
\usepackage{amsthm}
\usepackage{color}
\usepackage{enumitem}
\usepackage{tikz}
\usepackage{appendix}
\usepackage[colorlinks=true]{hyperref}
\usepackage[hyphenbreaks]{breakurl}
\input diagxy

\theoremstyle{plain}
  \newtheorem{thm}{Theorem}[section]
  
  \newtheorem{prop}[thm]{Proposition}
  \newtheorem{cor}[thm]{Corollary}
\theoremstyle{definition}
  \newtheorem{defn}[thm]{Definition}
  
  \newtheorem{rem}[thm]{Remark}

\DeclareMathAlphabet{\mathcal}{OMS}{cmsy}{m}{n}

\DeclareMathOperator{\dom}{dom}
\DeclareMathOperator{\cod}{cod}

\DeclareMathOperator{\id}{id}
\DeclareMathOperator{\ob}{ob}

\DeclareMathOperator{\sIm}{{\sf Im}}

\makeatletter
\def\ps@pprintTitle{%
 \let\@oddhead\@empty
  \let\@evenhead\@empty
  \def\@oddfoot{\vbox{\hsize=\textwidth\footnotesize
  \vskip 8pt
  \copyright 2020. This manuscript version is made available under the CC-BY-NC-ND 4.0 license \url{https://creativecommons.org/licenses/by-nc-nd/4.0/}. The published version is available at \url{https://doi.org/10.1016/j.fss.2020.10.015}.\\
  }}%
  \let\@evenfoot\@oddfoot}
\makeatother

\def\oto{{\bfig\morphism<180,0>[\mkern-4mu`\mkern-4mu;]\place(86,0)[\circ]\efig}}
\def\rto{{\bfig\morphism<180,0>[\mkern-4mu`\mkern-4mu;]\place(82,0)[\mapstochar]\efig}}

\newcommand{\da}{\downarrow}

\newcommand{\ra}{\rightarrow}
\newcommand{\la}{\leftarrow}

\newcommand{\lra}{\longrightarrow}
\newcommand{\lda}{\swarrow}
\newcommand{\rda}{\searrow}

\newcommand{\bv}{\bigvee}
\newcommand{\bw}{\bigwedge}

\newcommand{\dv}{\dashv}

\newcommand{\nat}{\natural}

\renewcommand{\phi}{\varphi}
\newcommand{\al}{\alpha}
\newcommand{\be}{\beta}

\newcommand{\ga}{\gamma}

\newcommand{\lam}{\lambda}

\newcommand{\Om}{\Omega}

\newcommand{\ka}{\kappa}

\newcommand{\CC}{\mathcal{C}}

\newcommand{\CO}{\mathcal{O}}

\newcommand{\CQ}{\mathcal{Q}}

\newcommand{\sF}{{\sf F}}

\newcommand{\sI}{{\sf I}}

\newcommand{\sK}{{\sf K}}

\newcommand{\sP}{{\sf P}}
\newcommand{\sQ}{{\sf Q}}

\newcommand{\si}{a}
\newcommand{\sj}{b}
\newcommand{\sd}{{\sf d}}

\newcommand{\so}{{\sf o}}

\newcommand{\sy}{{\sf y}}

\newcommand{\cl}{{\sf cl}}

\newcommand{\BB}{{\bf B}}
\newcommand{\BD}{{\bf D}}

\newcommand{\FD}{\mathfrak{D}}

\newcommand{\Arr}{{\bf Arr}}
\newcommand{\Cat}{{\bf Cat}}

\newcommand{\ClsDist}{{\bf ClsDist}}
\newcommand{\ClsCloRel}{{\bf ClsRel}_{\cl}}

\newcommand{\CCat}{{\bf CCat}}
\newcommand{\Chu}{{\bf Chu}}
\newcommand{\ChuCon}{{\bf ChuCon}}

\newcommand{\Dist}{{\bf Dist}}
\newcommand{\Inf}{{\bf Inf}}

\newcommand{\Int}{{\bf Int}}
\newcommand{\IntDist}{{\bf IntDist}}
\newcommand{\IntRel}{{\bf IntRel}}
\newcommand{\TopRel}{{\bf TopRel}}

\newcommand{\Rel}{{\bf Rel}}

\newcommand{\Sup}{{\bf Sup}}

\newcommand{\QCat}{\CQ\text{-}\Cat}

\newcommand{\QCCat}{\CQ\text{-}\CCat}

\newcommand{\QDist}{\CQ\text{-}\Dist}
\newcommand{\QChu}{\CQ\text{-}\Chu}
\newcommand{\QInf}{\CQ\text{-}\Inf}

\newcommand{\QInt}{\CQ\text{-}\Int}
\newcommand{\QIntDist}{\CQ\text{-}\IntDist}
\newcommand{\QIntRel}{\CQ\text{-}\IntRel}
\newcommand{\QClsDist}{\CQ\text{-}\ClsDist}
\newcommand{\QClsCloDist}{(\QClsDist)_{\cl}}

\newcommand{\QRel}{\CQ\text{-}\Rel}
\newcommand{\QSup}{\CQ\text{-}\Sup}

\newcommand{\tphi}{\widetilde{\phi}}

\newcommand{\tpsi}{\widetilde{\psi}}

\newcommand{\hK}{\hat{\sK}}
\newcommand{\hI}{\hat{\sI}}
\newcommand{\hKd}{\hK_{\sd}}
\newcommand{\hId}{\hI_{\sd}}
\newcommand{\hIo}{\hI_{\so}}

\newcommand{\sPd}{\sP^{\dag}}

\newcommand{\co}{{\rm co}}
\newcommand{\op}{{\rm op}}

\newcommand{\PX}{\sP X}
\newcommand{\PY}{\sP Y}

\newcommand{\PdX}{\sPd X}

\newcommand{\BQ}{\BB(\CQ)}

\newcommand{\DQ}{\BD(\CQ)}

\newcommand{\ArrQ}{\Arr(\CQ)}

\renewcommand{\leq}{\leqslant}
\renewcommand{\geq}{\geqslant}

\newcommand{\RQDist}{\BD(\QDist)_{\rm reg}}
\newcommand{\QCD}{(\QSup)_{\rm ccd}}

\numberwithin{equation}{section}

\begin{document}

\begin{frontmatter}



\title{Diagonals between $\mathcal{Q}$-distributors}


\author{Lili Shen}
\ead{shenlili@scu.edu.cn}

\address{School of Mathematics, Sichuan University, Chengdu 610064, China}

\begin{abstract}
For a small quantaloid $\mathcal{Q}$, it is shown that the category of $\mathcal{Q}$-distributors and diagonals is equivalent to a quotient category of the category of $\mathcal{Q}$-interior spaces and continuous $\mathcal{Q}$-distributors. Kan adjunctions induced by $\mathcal{Q}$-distributors play a crucial role in establishing this equivalence.
\end{abstract}

\begin{keyword}
Diagonal \sep Quantaloid \sep $\mathcal{Q}$-distributor \sep $\mathcal{Q}$-interior space \sep Continuous $\mathcal{Q}$-distributor \sep Kan adjunction

\MSC[2020] 18D20 \sep 18D60 \sep 18A40
\end{keyword}

\end{frontmatter}




\section{Introduction}

Given a \emph{(unital) quantale} \cite{Rosenthal1990} $\sQ$, the category $\BD(\sQ)$ of \emph{diagonals} \cite{Hoehle2011a,Pu2012,Stubbe2014} of $\sQ$ has been extensively studied in the fuzzy community; see, e.g., \cite{Li2017,Pu2012,Tao2014,Hoehle2015,GutierrezGarcia2018,He2019,Lai2020}. More specifically, $\BD(\sQ)$ is a \emph{quantaloid} \cite{Rosenthal1996,Stubbe2005,Stubbe2006,Stubbe2014}, and categories enriched in $\BD(\sQ)$ are precisely \emph{$\sQ$-preordered $\sQ$-subsets} \cite{GutierrezGarcia2018}. In particular, if $\Om$ is a frame, then \emph{symmetric} $\BD(\Om)$-categories are \emph{$\Om$-sets} \cite{Fourman1979}, and $\BD(\Om)$-categories are \emph{skew $\Om$-sets} \cite{Borceux1998}; if $[0,\infty]$ is the \emph{Lawvere quantale} \cite{Lawvere1973}, then $\BD[0,\infty]$-categories are \emph{(generalized) partial metric spaces} \cite{Matthews1994,Bukatin2009,Hoehle2011a,Pu2012,Stubbe2014}.

In fact, the construction of diagonals is well known in category theory. From each category $\CC$ we may construct the \emph{arrow category} $\Arr(\CC)$ \cite{MacLane1998} of $\CC$, whose objects are $\CC$-arrows and whose morphisms from $u$ to $v$ are pairs $(s:\dom u\to\dom v,\ t:\cod u\to\cod v)$ of $\CC$-arrows such that the square
$$\bfig
\square<500,400>[\bullet`\bullet`\bullet`\bullet;s`u`v`t]
\efig$$
is commutative. There is a congruence on $\Arr(\CC)$, given by $(s,t)\sim(s',t')$ if the commutative squares
$$\bfig
\square<500,400>[\bullet`\bullet`\bullet`\bullet;s`u`v`t]
\square(1000,0)<500,400>[\bullet`\bullet`\bullet`\bullet;s'`u`v`t']
\morphism(0,400)/-->/<500,-400>[\bullet`\bullet;]
\morphism(1000,400)/-->/<500,-400>[\bullet`\bullet;]
\efig$$
have the same \emph{diagonal}; that is, if $v\circ s=t\circ u=v\circ s'=t'\circ u$. The induced quotient category
$$\BD(\CC):=\Arr(\CC)/\!\sim$$
is precisely the \emph{Freyd completion} \cite{Freyd1966,Grandis2000,Grandis2002} of $\CC$. In a nutshell, the category of diagonals of a category $\CC$ is the Freyd completion of $\CC$, and it has received considerable attention in the realm of category theory as well \cite{Higgs1970,Fourman1979,Walters1981,Higgs1984,Shen2016a,Hofmann2018a,Tholen2019}.

\emph{Distributors} \cite{Benabou2000,Borceux1994a,Borceux1994b} (also \emph{profunctors} or \emph{bimodules}) generalize functors in the same way as relations generalize maps. For a \emph{small} quantaloid $\CQ$, the category $\QDist$ of $\CQ$-categories and $\CQ$-distributors is again a (large) quantaloid. The aim of this paper is to investigate the category
$$\BD(\QDist)$$
of diagonals of the quantaloid $\QDist$.

In order to explain the motivation of the results of this paper, let us recall that a \emph{Chu transform} \cite{Shen2015,Shen2016} (also \emph{infomorphism} \cite{Shen2014,Shen2013a})
$$(f,g):(\phi:X\oto Y)\to(\psi:X'\oto Y')$$
between $\CQ$-distributors is a pair of $\CQ$-functors $f:X\to X'$ and $g:Y'\to Y$, such that
\begin{equation} \label{Chu-graph}
\psi\circ f_{\nat}=g^{\nat}\circ\phi,\quad\text{or equivalently,}\quad\psi\lda f^{\nat}=g_{\nat}\rda\phi,
\end{equation}
where $f_{\nat}$ and $f^{\nat}$ are the \emph{graph} and \emph{cograph} of $f$, respectively, and $\lda$, $\rda$ are \emph{left} and \emph{right implications} in $\QDist$, respectively. The two equivalent characterizations of Chu transforms in \eqref{Chu-graph} allow us to extend the category $\QChu$ of $\CQ$-distributors and Chu transforms in two directions (see \cite[Proposition 3.2.1]{Shen2016a}):
\begin{equation} \label{QChu-Arr-D-ChuCon-B}
\bfig
\Atriangle/->`->`/<800,300>[\QChu`\Arr(\QDist)`\ChuCon(\QDist)^{\op};(\Box_{\nat},\Box^{\nat})`(\Box^{\nat},\Box_{\nat})`]
\square(0,-300)/`->`->`/<1600,300>[\Arr(\QDist)`\ChuCon(\QDist)^{\op}`\BD(\QDist)`\BB(\QDist)^{\op};```]
\efig
\end{equation}
In the above diagram, $\ChuCon(\QDist)$ and $\BB(\QDist)$, called the categories of \emph{Chu connections} and \emph{back diagonals} \cite{Shen2016a} of $\QDist$, dualize the constructions of $\Arr(\QDist)$ and $\BD(\QDist)$, respectively. In fact, all the categories in \eqref{QChu-Arr-D-ChuCon-B}, expect $\QChu$, are actually quantaloids. It is already known that
\begin{itemize}
\item $\BB(\QDist)$ is dually equivalent to the quantaloid $\QSup$($=\QCCat$) of separated complete $\CQ$-categories and left adjoint $\CQ$-functors \cite{Shen2016a}, and
\item $\QSup$ is dually equivalent to the quantaloid $\QClsCloDist$ of $\CQ$-closure spaces and closed continuous $\CQ$-distributors \cite{Shen2016b}.
\end{itemize}
Hence, the combination of the main results of \cite{Shen2016b,Shen2016a} renders equivalences of quantaloids
\begin{equation} \label{BQDist-QSup-QClsCloDist}
\BB(\QDist)^{\op}\simeq\QSup\simeq\QClsCloDist^{\op},
\end{equation}
which unveil the categorical and topological nature of \emph{back diagonals between $\CQ$-distributors}.

Since $\BB(\QDist)$ may be considered as a dualization of $\BD(\QDist)$ (cf. \cite[Subsection 1.1]{Shen2016a}), it is natural to ask whether similar equivalences of categories could be established for $\BD(\QDist)$. In other words, is it possible to find any categorical or topological interpretation of \emph{diagonals between $\CQ$-distributors}?

Unfortunately, if we take a closer look at the difference between $\BD(\QDist)$ and $\BB(\QDist)$, we would see that neither $\QSup$ nor its dual construction $\QInf$ (the category of separated complete $\CQ$-categories and right adjoint $\CQ$-functors) can be (dually) equivalent to $\BD(\QDist)$:
\begin{enumerate}[label=(\arabic*)]
\item \label{DQDist-QSup-why-1} The canonical functor $\ChuCon(\QDist)^{\op}\to\QSup$ that leads to the equivalence $\BB(\QDist)^{\op}\simeq\QSup$ (see \cite[Proposition 3.3.1]{Shen2016a}) is constructed through fixed points of \emph{Isbell adjunctions} \cite{Shen2013a}, while the parallel functor $\Arr(\QDist)^{\op}\to\QSup$ (see \cite[Proposition 5.1]{Lai2018a}) is constructed through fixed points of \emph{Kan adjunctions} \cite{Shen2013a}.
\item \label{DQDist-QSup-why-2} In the special case that $\sQ$ is a commutative and integral quantale, it is already known from \cite{Lai2009} that every complete $\sQ$-category is isomorphic to the $\sQ$-category of fixed points of some Isbell adjunction, but it does not hold for Kan adjunction. Indeed, \cite[Theorem 5.3]{Lai2009} actually states that every complete $\sQ$-category is isomorphic to the $\sQ$-category of fixed points of some Kan adjunction if, and only if, $\sQ$ is a \emph{Girard quantale} \cite{Rosenthal1990,Yetter1990}.
\end{enumerate}
As a result of \ref{DQDist-QSup-why-1} and \ref{DQDist-QSup-why-2}, the canonical functor from $\BD(\QDist)$ to $\QSup$ (or $\QInf$) cannot be essentially surjective on objects, and thus it cannot be an equivalence of categories. As a compromise, it is proved in \cite[Theorem 5.5]{Lai2018a} that there is an equivalence of quantaloids
\begin{equation} \label{RQDist-QCD}
\RQDist^{\op}\simeq\QCD,
\end{equation}
where $\RQDist$ is the full subquantaloid of $\BD(\QDist)$ with objects restricting to \emph{regular} $\CQ$-distributors, and $\QCD$ is the full subquantaloid of $\QSup$ with objects restricting to \emph{completely distributive} $\CQ$-categories. However, the equivalence \eqref{RQDist-QCD} cannot be extended to $\BD(\QDist)^{\op}$ and $\QSup$.

In spite of the difficulty in searching for the \emph{categorical} meaning of diagonals through complete $\CQ$-categories, in this paper we manage to present a \emph{topological} interpretation of diagonals via \emph{$\CQ$-interior spaces}. As a dual notion of $\CQ$-closure space \cite{Shen2016b}, a $\CQ$-interior space $(X,\si)$ is defined as a $\CQ$-category $X$ equipped with a \emph{$\CQ$-interior operator} \cite{Shen2013a} $\si$ on its presheaf $\CQ$-category $\PX$. By constructing an adjunction
$$\sK\dv\sI:\QInt\to\QChu,$$
it is shown in Section \ref{Q-Int} that the category $\QInt$ of $\CQ$-interior spaces and \emph{continuous $\CQ$-functors} is a retract and coreflective subcategory of $\QChu$ (Theorem \ref{I-dv-K}). Explicitly, the functor $\sK$ sends each $\CQ$-distributor $\phi:X\oto Y$ to the $\CQ$-interior space $(X,\phi^*\phi_*)$, where
$$\phi^*\dv\phi_*:\PX\to\PY$$
is the \emph{Kan adjunction} \cite{Shen2013a} induced by $\phi$.

Since the continuity of a $\CQ$-functor between $\CQ$-interior spaces is completely determined by its graph, it is natural to formulate the notion of \emph{continuous $\CQ$-distributor}; see Definition \ref{cont-dist-def}. In Section \ref{Q-IntDist} we construct a full functor
$$\hK:\Arr(\QDist)\to\QIntDist,$$
which coincides with $\sK$ on objects and has a right inverse $\hI:\QIntDist\to\Arr(\QDist)$; hence, $\QIntDist$ is a retract of $\Arr(\QDist)$ (Proposition \ref{hK-hI-id}).

In fact, $\QIntDist$ is a quantaloid, and there is a congruence on $\QIntDist$, given by $\zeta\sim\zeta':(X,\si)\oto(Y,\sj)$ if
$$\zeta^*\sj=\zeta'^*\sj,$$
which intuitively identifies ``continuous maps that are indistinguishable by preimages of open sets'', and we denote the quotient quantaloid by $(\QIntDist)_{\so}$. The main result of this paper, Theorem \ref{hKd-hId-equiv}, gives an equivalence of quantaloids
\begin{equation} \label{DQDist-QIntDisto}
\BD(\QDist)\simeq(\QIntDist)_{\so}.
\end{equation}
Therefore, from the topological point of view, we may conclude that a diagonal between $\CQ$-distributors is essentially an equivalence class of continuous $\CQ$-distributors between $\CQ$-interior spaces.

Moreover, we establish the discrete version of the equivalence \eqref{DQDist-QIntDisto} in Section \ref{Q-IntRel} (Theorem \ref{DRel-QIntRelo-equiv}), which is in particular applied to (classical) interior spaces and topological spaces (Corollaries \ref{Rel-IntRelo-equiv} and \ref{Relf-TopRelo-equiv}). Finally, in Section \ref{Q-Girard} we discuss a special case, i.e., when $\CQ$ is a Girard quantaloid. In this case, we have an isomorphism of quantaloids
$$\BD(\QDist)\cong\BB(\QDist)$$
by Propositions \ref{Girard-Arr-ChuCon} and \ref{QDist-Girard}. Consequently, the equivalences \eqref{BQDist-QSup-QClsCloDist} and \eqref{DQDist-QIntDisto} are combined to
$$\BD(\QDist)\simeq\BB(\QDist)\simeq(\QIntDist)_{\so}\simeq\QClsCloDist\simeq(\QSup)^{\op}$$
if $\CQ$ is Girard (Theorem \ref{Girard-equiv}).

\section{Diagonals of a quantaloid} \label{Diagonals}

A \emph{quantaloid} \cite{Rosenthal1996} is a category enriched in the symmetric monoidal closed category $\Sup$. Explicitly, a quantaloid $\CQ$ is a (possibly large) 2-category with its 2-cells given by order, such that each hom-set $\CQ(p,q)$ is a complete lattice and the composition $\circ$ of $\CQ$-arrows preserves joins on both sides, i.e.,
$$v\circ\Big(\bv_{i\in I} u_i\Big)=\bv_{i\in I}v\circ u_i,\quad\Big(\bv_{i\in I} v_i\Big)\circ u=\bv_{i\in I}v_i\circ u$$
for all $\CQ$-arrows $u,u_i:p\to q$ and $v,v_i:q\to r$ $(i\in I)$. Hence, $\CQ$ has ``internal homs'', denoted by $\lda$ and $\rda\,$, as the right adjoints of the composition maps:
$$(-\circ u)\dv(-\lda u):\CQ(p,r)\to\CQ(q,r)\quad\text{and}\quad (v\circ -)\dv(v\rda -):\CQ(p,r)\to\CQ(p,q);$$
explicitly,
$$v\circ u\leq w\iff v\leq w\lda u\iff u\leq v\rda w$$
for all $\CQ$-arrows $u:p\to q$, $v:q\to r$, $w:p\to r$.

A \emph{homomorphism} of quantaloids is a functor between the underlying categories that preserves joins of arrows. A homomorphism of quantaloids is \emph{full} (resp. \emph{faithful}, an \emph{equivalence} of quantaloids, an \emph{isomorphism} of quantaloids) if the underlying functor is full (resp. faithful, an equivalence of underlying categories, an isomorphism of underlying categories).

Each quantaloid $\CQ$ induces an arrow category $\ArrQ$ of $\CQ$ with $\CQ$-arrows as objects and pairs ($s:p\to p',\ t:q\to q'$) of $\CQ$-arrows satisfying
$$v\circ s=t\circ u$$
$$\bfig
\square<500,400>[p`p'`q`q';s`u`v`t]
\efig$$
as arrows from $u:p\to q$ to $v:p'\to q'$. $\ArrQ$ is again a quantaloid with the componentwise local order inherited from $\CQ$.

A \emph{congruence} $\vartheta$ on a quantaloid $\CQ$ consists of a family of equivalence relations $\vartheta_{p,q}$ on each hom-set $\CQ(p,q)$ that is compatible with compositions and joins of $\CQ$-arrows, i.e., $(v\circ u,v'\circ u')\in\vartheta_{p,r}$ whenever $(u,u')\in\vartheta_{p,q}$, $(v,v')\in\vartheta_{q,r}$ and $\Big(\displaystyle\bv_{i\in I}u_i,\displaystyle\bv_{i\in I}u'_i\Big)\in\vartheta_{p,q}$ whenever $(u_i,u'_i)\in\vartheta_{p,q}$ for all $i\in I$.

Each congruence $\vartheta$ on $\CQ$ induces a \emph{quotient quantaloid} $\CQ/\vartheta$ equipped with the same objects as $\CQ$. Compositions and joins of arrows in $\CQ/\vartheta$ are clearly well defined, and the obvious quotient functor
$$\CQ\to\CQ/\vartheta$$
is a full quantaloid homomorphism.

For arrows $(s,t),(s',t'):u\to v$ in $\ArrQ$, we denote by $(s,t)\sim(s',t')$ if the commutative squares
$$\bfig
\square<500,400>[p`p'`q`q';s`u`v`t]
\square(1000,0)<500,400>[p`p'`q`q';s'`u`v`t']
\morphism(0,400)/-->/<500,-400>[p`q';]
\morphism(1000,400)/-->/<500,-400>[p`q';]
\efig$$
have the same \emph{diagonal}; that is, if
$$v\circ s=t\circ u=v\circ s'=t'\circ u.$$
``$\sim$'' gives rise to a congruence on $\ArrQ$; the induced quotient quantaloid, denoted by
$$\DQ:=\ArrQ/\!\sim,$$
is called the quantaloid of \emph{diagonals} \cite{Stubbe2014} of $\CQ$, which is precisely the \emph{Freyd completion} \cite{Freyd1966,Grandis2000,Grandis2002} of $\CQ$.

\section{Quantaloid-enriched categories and their distributors} \label{Q-Categories}

From now on, we let $\CQ$ be a \emph{small} quantaloid, i.e., $\ob\CQ$ is a set. A \emph{$\CQ$-category} \cite{Rosenthal1996,Stubbe2005} consists of a set $X$ over $\ob\CQ$, i.e., a set $X$ equipped with a \emph{type} map $|\text{-}|:X\to\ob\CQ$, and hom-arrows $\al(x,y)\in\CQ(|x|,|y|)$, such that
$$1_{|x|}\leq\al(x,y)\quad\text{and}\quad\al(y,z)\circ\al(x,y)\leq\al(x,z)$$
for all $x,y,z\in X$. If $(X,\al)$ is a $\CQ$-category, each subset $Y\subseteq X$ is equipped with the (full) \emph{$\CQ$-subcategory} structure inherited from $\al$.

Each $\CQ$-category $(X,\al)$ is endowed with an underlying (pre)order given by
$$x\leq y\iff|x|=|y|\ \text{and}\ 1_{|x|}\leq\al(x,y),$$
and we write $x\cong y$ if $x\leq y$ and $y\leq x$. We say that $(X,\al)$ is \emph{separated} (or \emph{skeletal}) if $x=y$ whenever $x\cong y$ in $X$.

A \emph{$\CQ$-functor} (resp. \emph{fully faithful} $\CQ$-functor) $f:(X,\al)\to(Y,\be)$ between $\CQ$-categories is a map $f:X\to Y$ such that
$$|x|=|fx|\quad\text{and}\quad\al(x,y)\leq\be(fx,fy)\quad(\text{resp.}\ \al(x,y)=\be(fx,fy))$$
for all $x,y\in X$. With the pointwise order of $\CQ$-functors given by
$$f\leq g:(X,\al)\to(Y,\be)\iff\forall x\in X:\ fx\leq gx,$$
$\CQ$-categories and $\CQ$-functors constitute a locally ordered 2-category $\QCat$. A pair of $\CQ$-functors $f:(X,\al)\to(Y,\be)$, $g:(Y,\be)\to(X,\al)$ forms an adjunction $f\dv g$ in $\QCat$ if $\be(fx,y)=\al(x,gy)$ for all $x\in X$, $y\in Y$. 

A \emph{$\CQ$-distributor} $\phi:(X,\al)\oto(Y,\be)$ between $\CQ$-categories is a map that assigns to each pair $(x,y)\in X\times Y$ a $\CQ$-arrow $\phi(x,y)\in\CQ(|x|,|y|)$, such that $$\be(y,y')\circ\phi(x,y)\circ\al(x',x)\leq\phi(x',y')$$
for all $x,x'\in X$, $y,y'\in Y$. With the pointwise order inherited from $\CQ$, the locally ordered 2-category $\QDist$ of $\CQ$-categories and $\CQ$-distributors becomes a (large) quantaloid in which
\begin{align*}
&\psi\circ\phi:(X,\al)\oto(Z,\ga),\quad(\psi\circ\phi)(x,z)=\bv_{y\in Y}\psi(y,z)\circ\phi(x,y),\\
&\xi\lda\phi:(Y,\be)\oto(Z,\ga),\quad(\xi\lda\phi)(y,z)=\bw_{x\in X}\xi(x,z)\lda\phi(x,y),\\
&\psi\rda\xi:(X,\al)\oto(Y,\be),\quad (\psi\rda\xi)(x,y)=\bw_{z\in Z}\psi(y,z)\rda\xi(x,z)
\end{align*}
for all $\CQ$-distributors $\phi:(X,\al)\oto(Y,\be)$, $\psi:(Y,\be)\oto(Z,\ga)$, $\xi:(X,\al)\oto(Z,\ga)$; the identity $\CQ$-distributor on $(X,\al)$ is given by its hom $\al:(X,\al)\oto(X,\al)$. 

Each $\CQ$-functor $f:(X,\al)\to(Y,\be)$ induces an adjunction $f_{\nat}\dv f^{\nat}$ in $\QDist$ (i.e., $\al\leq f^{\nat}\circ f_{\nat}$ and $f_{\nat}\circ f^{\nat}\leq\be$) given by
\begin{align*}
&f_{\nat}:(X,\al)\oto(Y,\be),\quad f_{\nat}(x,y)=\be(fx,y),\\
&f^{\nat}:(Y,\be)\oto(X,\al),\quad f^{\nat}(y,x)=\be(y,fx),
\end{align*}
called the \emph{graph} and \emph{cograph} of $f$, respectively. Obviously, the identity $\CQ$-distributor $\al:(X,\al)\oto(X,\al)$ is the cograph of the identity $\CQ$-functor $1_X:X\to X$. Hence, in what follows
$$1_X^{\nat}=\al$$
will be our standard notation for the hom of a $\CQ$-category $X=(X,\al)$ if no confusion arises. It is easy to see that
\begin{equation} \label{f-leq-g-graph}
f\leq g:X\to Y\iff g_{\nat}\leq f_{\nat}:X\oto Y\iff f^{\nat}\leq g^{\nat}:Y\oto X,
\end{equation}
and thus both the graphs and cographs of $\CQ$-functors are 2-functorial as
$$(-)_{\nat}:(\QCat)^{\co}\to\QDist,\quad(-)^{\nat}:(\QCat)^{\op}\to\QDist,$$
where ``$\co$'' refers to the dualization of 2-cells.

For each $q\in\ob\CQ$, let $\{q\}$ denote the one-object $\CQ$-category whose only object has type $q$ and hom $1_q$. $\CQ$-distributors of the form $\mu:X\oto\{q\}$ are called \emph{presheaves} (of type $q$) on $X$, which constitute a separated $\CQ$-category $\PX$ with
$$1_{\PX}^{\nat}(\mu,\mu')=\mu'\lda\mu$$
for all $\mu,\mu'\in\PX$. The \emph{Yoneda embedding}
$$\sy_X:X\to\PX,\ x\mapsto 1_X^{\nat}(-,x)$$
is clearly a fully faithful $\CQ$-functor. 

Each $\CQ$-distributor $\phi:X\oto Y$ induces a \emph{Kan adjunction} \cite{Shen2013a}
$$\phi^*\dv\phi_*:\PX\to\PY$$
in $\QCat$ with
$$\phi^*\lam=\lam\circ\phi\quad\text{and}\quad\phi_*\mu=\mu\lda\phi$$
for all $\lam\in\PY$, $\mu\in\PX$. Moreover, $(-)^*:\QDist\to(\QCat)^{\op}$ is 2-functorial and left adjoint to $(-)^{\nat}:(\QCat)^{\op}\to\QDist$ \cite{Heymans2010}, which gives rise to isomorphisms
$$\QDist(X,Y)\cong\QCat(Y,\PX)$$
for all $\CQ$-categories $X$, $Y$, and we denote by
\begin{equation} \label{tphi-def}
\tphi:Y\to\PX,\quad \tphi y=\phi(-,y)
\end{equation}
the transpose of each $\CQ$-distributor $\phi:X\oto Y$. It is easy to see that
\begin{equation} \label{tphi-Yoneda}
\tphi=\phi^*\sy_Y.
\end{equation}
In particular, each $\CQ$-functor $f:X\to Y$ induces an adjunction $f^{\ra}\dv f^{\la}$ in $\QCat$ with
\begin{equation} \label{fra-fla-def}
f^{\ra}:=(f^{\nat})^*:\PX\to\PY\quad\text{and}\quad f^{\la}:=(f_{\nat})^*=(f^{\nat})_*:\PY\to\PX,
\end{equation}
where $(f_{\nat})^*=(f^{\nat})_*$ may be easily verified by routine calculation.

\section{$\CQ$-interior spaces} \label{Q-Int}

A \emph{$\CQ$-interior space} is a pair $(X,\si)$ that consists of a $\CQ$-category $X$ and a $\CQ$-interior operator \cite{Shen2013a} $\si$ on $\PX$; that is, a $\CQ$-functor $\si:\PX\to\PX$ with
\begin{equation} \label{int-def}
\si\leq 1_{\PX}\quad\text{and}\quad\si\si=\si,
\end{equation}
where we write $\si\si=\si$ instead of $\si\si\cong\si$ because the presheaf $\CQ$-category $\PX$ is separated. We denote by
$$\CO(X,\si):=\{\mu\in\PX\mid\si\mu=\mu\}$$
the $\CQ$-subcategory of $\PX$ consisting of \emph{open} presheaves of $(X,\si)$.

\begin{rem} \label{int-TAC}
The definition of $\CQ$-interior space here deviates from that of \cite{Lai2017b}, in which a $\CQ$-interior space is defined as a pair $(X,c)$, with $c$ being a \emph{$\CQ$-closure operator} on the \emph{copresheaf $\CQ$-category} $\PdX$ of $X$. In the case that $\CQ$ is a \emph{commutative quantale} \cite{Rosenthal1990} and $X$ is a \emph{discrete} $\CQ$-category (i.e., a set), $\CQ$-interior operators on $\PX$ are essentially the same as $\CQ$-closure operators on $\PdX$; however, it should be noted that they may not coincide when $\CQ$ is a general quantaloid. 
\end{rem}

A \emph{continuous $\CQ$-functor} $f:(X,\si)\to(Y,\sj)$ between $\CQ$-interior spaces is a $\CQ$-functor $f:X\to Y$ such that
$$f^{\la}\sj\leq \si f^{\la}:\PY\to\PX.$$
$\CQ$-interior spaces and continuous $\CQ$-functors constitute a 2-category $\QInt$, with the local order inherited from $\QCat$.

The following proposition shows that continuous $\CQ$-functors may be characterized as preimages of open presheaves staying open, and we will prove its generalized version in the next section (see Proposition \ref{cont-dist}):

\begin{prop} \label{cont-functor}
Let $(X,\si)$, $(Y,\sj)$ be $\CQ$-interior spaces. For each $\CQ$-functor $f:X\to Y$, the following statements are equivalent:
\begin{enumerate}[label={\rm(\roman*)}]
\item \label{cont-functor:def} $f:(X,\si)\to(Y,\sj)$ is a continuous $\CQ$-functor.
\item \label{cont-functor:l} $f^{\la}\sj\leq\si f^{\la}\sj$, thus $f^{\la}\sj=\si f^{\la}\sj$; that is, $f^{\la}\lam\in\CO(X,\si)$ whenever $\lam\in\CO(Y,\sj)$.
\item \label{cont-functor:r} $\sj(f_{\nat})_*\leq\sj(f_{\nat})_*\si$, thus $\sj(f_{\nat})_*=\sj(f_{\nat})_*\si$.
\end{enumerate}
\end{prop}

Recall that a \emph{Chu transform} \cite{Shen2015,Shen2016} (called \emph{infomorphism} in \cite{Shen2014,Shen2013a})
$$(f,g):(\phi:X\oto Y)\to(\psi:X'\oto Y')$$
between $\CQ$-distributors is a pair of $\CQ$-functors $f:X\to X'$ and $g:Y'\to Y$, such that $\psi\circ f_{\nat}=g^{\nat}\circ\phi$, or equivalently, $\psi(f-,-)=\phi(-,g-)$.
$$\bfig
\square/->`->`<-`->/<500,400>[X`Y`X'`Y';\phi`f`g`\psi]
\place(250,0)[\circ]
\place(250,400)[\circ]
\square(1000,0)<500,400>[X`Y`X'`Y';\phi`f_\nat`g^\nat`\psi]
\place(1250,0)[\circ]
\place(1250,400)[\circ]
\place(1000,200)[\circ]
\place(1500,200)[\circ]
\efig$$
With Chu transforms being ordered as
$$(f,g)\leq(f',g'):\phi\to\psi\iff f\leq f'\ \text{and}\ g\geq g',$$
we obtain a locally ordered 2-category $\QChu$ of $\CQ$-distributors and Chu transforms.

Note that the Kan adjunction $\phi^*\dv\phi_*:\PX\to\PY$ induced by each $\CQ$-distributor $\phi:X\oto Y$ gives rise to a $\CQ$-interior operator $\phi^*\phi_*:\PX\to\PX$, and thus to a $\CQ$-interior space $(X,\phi^*\phi_*)$. The assignment
$$(\phi:X\oto Y)\mapsto(X,\phi^*\phi_*)$$
is in fact functorial from $\QChu$ to $\QInt$:

\begin{prop} \label{K-functor}
Let $(f,g):\phi\to\psi$ be a Chu transform between $\CQ$-distributors $\phi:X\oto Y$ and $\psi:X'\oto Y'$. Then $f:(X,\phi^*\phi_*)\to(X',\psi^*\psi_*)$ is a continuous $\CQ$-functor.
\end{prop}

\begin{proof}
In order to prove $f^{\la}\psi^*\psi_*\leq\phi^*\phi_* f^{\la}$, let us consider the following diagram:
$$\bfig
\iiixii|aalrrbb|<700,500>[\PX'`\PY'`\PX'`\PX`\PY`\PX;\psi_*`\psi^*`f^{\la}`g^{\ra}`f^{\la}`\phi_*`\phi^*]
\place(350,250)[\geq]
\efig$$
Since $f^{\ra}=(f_{\nat})^*$ and $g^{\ra}=(g^{\nat})^*$, one may check that the right square is commutative and that $g^{\ra}\psi_*\leq\phi_* f^{\la}$ similarly as in the proof of Proposition \ref{hK-functor}, by just trading $\zeta$ and $\eta$ there for $f_{\nat}$ and $g^{\nat}$, respectively. The details are left to the readers.
\end{proof}

From Proposition \ref{K-functor} we obtain a 2-functor
$$\sK:\QChu\to\QInt$$
that sends each Chu transform $(f,g):(\phi:X\oto Y)\to(\psi:X'\oto Y')$ to the continuous $\CQ$-functor 
$$f:(X,\phi^*\phi_*)\to(X',\psi^*\psi_*).$$

Conversely, from each $\CQ$-interior space $(X,\si)$ we may construct a $\CQ$-distributor
\begin{equation} \label{ka-def}
\ka_{\si}:X\oto\CO(X,\si),\quad\ka_{\si}(x,\mu)=\mu(x).
\end{equation}

\begin{prop} \label{I-ff}
A $\CQ$-functor $f:(X,\si)\to(Y,\sj)$ between $\CQ$-interior spaces is continuous if, and only if, there exists a (necessarily unique) $\CQ$-functor $g:\CO(Y,\sj)\to\CO(X,\si)$ such that
$$(f,g):(\ka_{\si}:X\oto\CO(X,\si))\to(\ka_{\sj}:Y\oto\CO(Y,\sj))$$
is a Chu transform.
\end{prop}

\begin{proof}
Suppose that $f:(X,\si)\to(Y,\sj)$ is continuous. From Proposition \ref{cont-functor}\ref{cont-functor:l} we know that the $\CQ$-functor $f^{\la}:\PY\to\PX$ can be restricted to
$$f^{\la}|_{\CO(Y,\sj)}:\CO(Y,\sj)\to\CO(X,\si),$$
and consequently
$$\ka_{\si}(x,f^{\la}|_{\CO(Y,\sj)}\lam)=(f^{\la}\lam)(x)=\lam(fx)=\ka_{\sj}(fx,\lam)$$
for all $x\in X$, $\lam\in\CO(Y,\sj)$, where the second equality holds because
\begin{equation} \label{lam-fx}
(f^{\la}\lam)(x)=((f_{\nat})^*\lam)(x)=\lam\circ f_{\nat}(x,-)=\lam\circ 1_Y^{\nat}(fx,-)=\lam(fx)
\end{equation}
for all $\lam\in\PY$, $x\in X$. Hence,
$$(f,f^{\la}|_{\CO(Y,\sj)}):\ka_{\si}\to\ka_{\sj}$$
is a Chu transform.

Conversely, suppose that $g:\CO(Y,\sj)\to\CO(X,\si)$ is a $\CQ$-functor making $(f,g):\ka_{\si}\to\ka_{\sj}$ a Chu transform. Then
$$(g\lam)(x)=\ka_{\si}(x,g\lam)=\ka_{\sj}(fx,\lam)=\lam(fx)=(f^{\la}\lam)(x)$$
for all $x\in X$, $\lam\in\CO(Y,\sj)$, where the last equality follows from \eqref{lam-fx}. This proves the uniqueness of $g$. In particular,
$$g=f^{\la}|_{\CO(Y,\sj)}:\CO(Y,\sj)\to\CO(X,\si)$$
means that $f^{\la}\lam\in\CO(X,\si)$ whenever $\lam\in\CO(Y,\sj)$. Hence, $f:(X,\si)\to(Y,\sj)$ is continuous by Proposition \ref{cont-functor}\ref{cont-functor:l}.

\end{proof}

By Proposition \ref{I-ff}, the assignment
$$(f:(X,\si)\to(Y,\sj))\mapsto((f,f^{\la}|_{\CO(Y,\sj)}):(\ka_{\si}:X\oto\CO(X,\si))\to(\ka_{\sj}:Y\oto\CO(Y,\sj)))$$
defines a fully faithful 2-functor
$$\sI:\QInt\to\QChu$$
that embeds $\QInt$ in $\QChu$ as a full 2-subcategory. In fact, this embedding is coreflective:

\begin{thm} \label{I-dv-K}
$\sK:\QChu\to\QInt$ is a left inverse and right adjoint of $\sI:\QInt\to\QChu$; hence, $\QInt$ is a retract and a coreflective 2-subcategory of $\QChu$.
\end{thm}

\begin{proof}
{\bf Step 1.} $\sK$ is a left inverse of $\sI$. For each $\CQ$-interior space $(X,\si)$, since $\sK\sI(X,\si)=(X,(\ka_{\si})^*(\ka_{\si})_*)$, we must show that
\begin{equation} \label{ka-si}
\si=(\ka_{\si})^*(\ka_{\si})_*.
\end{equation}
For each $\mu\in\PX$, $\lam\in\CO(X,\si)$, we claim that
\begin{equation} \label{si-lda-lam}
\mu\lda\lam=\si\mu\lda\lam,
\end{equation}
since the $\CQ$-functoriality of $\si$ forces
$$\mu\lda\lam=1_{\PX}^{\nat}(\lam,\mu)\leq 1_{\PX}^{\nat}(\si\lam,\si\mu)=1_{\PX}^{\nat}(\lam,\si\mu)=\si\mu\lda\lam,$$
and the reverse inequality follows from \eqref{int-def}. It follows that
\begin{align*}
\si\mu&=\ka_{\si}(-,\si\mu)&(\text{Equation \eqref{ka-def}})\\
&=1_{\CO(X,\si)}^{\nat}(-,\si\mu)\circ\ka_{\si}\\
&=\bv_{\lam\in\CO(X,\si)}1_{\CO(X,\si)}^{\nat}(\lam,\si\mu)\circ\ka_{\si}(-,\lam)\\
&=\bv_{\lam\in\CO(X,\si)}(\si\mu\lda\lam)\circ\ka_{\si}(-,\lam)\\
&=\bv_{\lam\in\CO(X,\si)}(\mu\lda\lam)\circ\ka_{\si}(-,\lam)&(\text{Equation \eqref{si-lda-lam}})\\
&=\bv_{\lam\in\CO(X,\si)}(\mu\lda\ka_{\si}(-,\lam))\circ\ka_{\si}(-,\lam)&(\text{Equation \eqref{ka-def}})\\
&=(\mu\lda\ka_{\si})\circ\ka_{\si}\\
&=(\ka_{\si})^*(\ka_{\si})_*\mu
\end{align*}
for all $\mu\in\PX$, as desired.


{\bf Step 2.} $\sK$ is a right adjoint of $\sI$. For each $\CQ$-interior space $(X,\si)$, since $\sK\sI(X,\si)=(X,\si)$, it suffices to show that the identity natural transformation
$$\{1_X:(X,\si)\to\sK\sI(X,\si)\mid(X,\si)\in\ob(\QInt)\}$$
is the unit of the adjunction $\sI\dv\sK$; that is, for each $\CQ$-distributor $\psi:Y\oto Z$ and continuous $\CQ$-functor
$$h:(X,\si)\to\sK(\psi:Y\oto Z)=(Y,\psi^*\psi_*),$$
there exists a unique Chu transform
$$(f,g):\sI(X,\si)=(\ka_{\si}:X\oto\CO(X,\si))\to(\psi:Y\oto Z)$$
such that the triangle
$$\bfig
\qtriangle<800,500>[(X,\si)`\sK\sI(X,\si)`(Y,\psi^*\psi_*);1_X`h`\sK(f,g)]
\efig$$
is commutative. Since $\sK(f,g)=f$, it remains to show that there exists a unique $\CQ$-functor $g:Z\to\CO(X,\si)$ such that
$$(h,g):(\ka_{\si}:X\oto\CO(X,\si))\to(\psi:Y\oto Z)$$
is a Chu transform. To this end, we define
$$g:=h^{\la}\tpsi:Z\to\PY\to\CO(X,\si).$$

First, $g$ is well defined. Indeed, for each $z\in Z$, by Equation \eqref{tphi-Yoneda} we have
$$\tpsi z=\psi^*\sy_Z z=\psi^*\psi_*\psi^*\sy_Z z\in\CO(Y,\psi^*\psi_*),$$
and together with the continuity of $h$ we deduce that $gz=h^{\la}\tpsi z\in\CO(X,\si)$.

Second, $(h,g)$ is a Chu transform. Indeed,
$$\ka_{\si}(x,gz)=(gz)(x)=(h^{\la}\tpsi z)(x)=(\tpsi z)(hx)=\psi(hx,z)$$
for all $x\in X$, $z\in Z$, where the penultimate equality follows from \eqref{lam-fx}.

Finally, for the uniqueness of $g$, suppose that $g':Z\to\CO(X,\si)$ is another $\CQ$-functor making $(h,g'):\ka_{\si}\to\psi$ a Chu transform. Then
$$(g'z)(x)=\ka_{\si}(x,g'z)=\psi(hx,z)=(\tpsi z)(hx)=(h^{\la}\tpsi z)(x)=(gz)(x)$$
for all $z\in Z$, $x\in X$. Hence $g'=g$.
\end{proof}

\begin{rem} \label{Pratt}
As pointed out by the anonymous referee, $\QChu$ may be identified with the \emph{comma category} (cf. \cite[Section II.6]{MacLane1998})
$$\QCat\da(-)^{\la},$$
where $(-)^{\la}:(\QCat)^{\op}\to\QCat$ is the functor sending each $\CQ$-functor $f:X\to Y$ to $f^{\la}:\PY\to\PX$ (cf. \eqref{fra-fla-def}). Under this identification, $\sI:\QInt\to\QChu$ actually sends each $\CQ$-interior space $(X,\si)$ to the inclusion $\CQ$-functor $\CO(X,\si)\ \to/^(->/\PX$; that is, the transpose $\widetilde{\ka_{\si}}$ of the $\CQ$-distributor $\ka_{\si}:X\oto\CO(X,\si)$.

Moreover, it is stated in \cite{Pratt1999} that a topological space is an \emph{extensional Chu space} whose columns are closed under arbitrary union and finite intersection. This observation may be generalized to our context as follows: a $\CQ$-distributor $\phi:X\oto Y$ can be identified with a $\CQ$-interior space if its transpose $\tphi:Y\to\PX$ is a fully faithful $\CQ$-functor, and if the image
$$\sIm\tphi:=\{\tphi y\mid y\in Y\}$$
of $\tphi$, as a $\CQ$-subcategory of $\PX$, is closed under suprema in $\PX$ (cf. \cite[Proposition 4.1.8]{Shen2014}); where the latter requirement may also be equivalently expressed as
$${\sup}_{\PX}\tphi^{\,\ra}\lam\in\sIm\tphi$$
for all $\lam\in\PY$. In this case, it is not difficult to prove that
$$\CO(X,\phi^*\phi_*)=\sIm\tphi,$$
so that the $\CQ$-interior space identified with $\phi:X\oto Y$ is precisely $(X,\phi^*\phi_*)$. 
\end{rem}

\section{Continuous $\CQ$-distributors} \label{Q-IntDist}

Since $f^{\la}=(f_{\nat})^*$, the continuity of a $\CQ$-functor $f:(X,\si)\to(Y,\sj)$ between $\CQ$-interior spaces is completely determined by its graph $f_{\nat}:X\oto Y$, i.e.,
$$(f_{\nat})^*\sj\leq\si (f_{\nat})^*:\PY\to\PX.$$
If $f_{\nat}$ is replaced by an arbitrary $\CQ$-distributor $\zeta:X\oto Y$, we come to the following definition:

\begin{defn} \label{cont-dist-def}
A \emph{continuous $\CQ$-distributor} $\zeta:(X,\si)\oto(Y,\sj)$ between $\CQ$-interior spaces is a $\CQ$-distributor $\zeta:X\oto Y$ such that
$$\zeta^*\sj\leq\si\zeta^*:\PY\to\PX.$$
\end{defn}

With the local order inherited from $\QDist$, $\CQ$-interior spaces and continuous $\CQ$-distributors constitute a (large) quantaloid $\QIntDist$, for it is easy to verify that compositions and joins of continuous $\CQ$-distributors are still continuous. There is clearly a 2-functor
$$(-)_{\nat}:(\QInt)^{\co}\to\QIntDist$$
sending each continuous $\CQ$-functor $f:(X,\si)\to(Y,\sj)$ to the continuous $\CQ$-distributor $f_{\nat}:(X,\si)\oto(Y,\sj)$.

\begin{prop} \label{cont-dist}
Let $(X,\si)$, $(Y,\sj)$ be $\CQ$-interior spaces. For each $\CQ$-distributor $\zeta:X\oto Y$, the following statements are equivalent:
\begin{enumerate}[label={\rm(\roman*)}]
\item \label{cont-dist:def} $\zeta:(X,\si)\oto(Y,\sj)$ is a continuous $\CQ$-distributor.
\item \label{cont-dist:l} $\zeta^*\sj\leq\si\zeta^*\sj$, thus $\zeta^*\sj=\si\zeta^*\sj$; that is, $\zeta^*\lam\in\CO(X,\si)$ whenever $\lam\in\CO(Y,\sj)$.
\item \label{cont-dist:r} $\sj\zeta_*\leq\sj\zeta_*\si$, thus $\sj\zeta_*=\sj\zeta_*\si$.
\end{enumerate}
\end{prop}

\begin{proof}
\ref{cont-dist:def}$\implies$\ref{cont-dist:l}: If $\zeta^*\sj\leq\si\zeta^*$, then $\zeta^*\sj=\zeta^*\sj\sj\leq\si\zeta^*\sj$.

\ref{cont-dist:l}$\implies$\ref{cont-dist:r}: This follows from $\sj=\sj\sj\leq\sj\zeta_*\zeta^*\sj\leq\sj\zeta_*\si\zeta^*\sj\leq\sj\zeta_*\si\zeta^*$ and $\zeta^*\dv\zeta_*$.

\ref{cont-dist:r}$\implies$\ref{cont-dist:def}: $\zeta^*\sj\leq\si\zeta^*$ follows immediately from $\sj\leq\sj\zeta_*\zeta^*\leq\sj\zeta_*\si\zeta^*\leq\zeta_*\si\zeta^*$ and $\zeta^*\dv\zeta_*$.
\end{proof}

As a generalized version of Proposition \ref{K-functor} we have:

\begin{prop} \label{hK-functor}
Each commutative square
$$\bfig
\square<500,400>[X`X'`Y`Y';\zeta`\phi`\psi`\eta]
\place(250,0)[\circ] \place(250,400)[\circ] \place(0,200)[\circ] \place(500,200)[\circ]
\efig$$
in $\QDist$ induces a continuous $\CQ$-distributor $\zeta:(X,\phi^*\phi_*)\oto(X',\psi^*\psi_*)$.
\end{prop}

\begin{proof}
In order to prove $\zeta^*\psi^*\psi_*\leq\phi^*\phi_*\zeta^*$, let us consider the following diagram:
$$\bfig
\iiixii|aalrrbb|<700,500>[\PX'`\PY'`\PX'`\PX`\PY`\PX;\psi_*`\psi^*`\zeta^*`\eta^*`\zeta^*`\phi_*`\phi^*]
\place(350,250)[\geq]
\efig$$
The commutativity of the right square follows immediately from the functoriality of $(-)^*:(\QDist)^{\op}\to\QCat$, and it remains to verify $\eta^*\psi_*\leq\phi_*\zeta^*$. Indeed,
\begin{align*}
\eta^*\psi_*\mu'&=(\mu'\lda\psi)\circ\eta\\
&\leq((\mu'\lda\psi)\circ\eta\circ\phi)\lda\phi\\
&=((\mu'\lda\psi)\circ\psi\circ\zeta)\lda\phi\\
&\leq(\mu'\circ\zeta)\lda\phi\\
&=\phi_*\zeta^*\mu'
\end{align*}
for all $\mu'\in\PX'$, and thus the conclusion follows.
\end{proof}

Proposition \ref{hK-functor} actually gives rise to a quantaloid homomorphism
$$\hK:\Arr(\QDist)\to\QIntDist$$
that sends each arrow $(\zeta,\eta):(\phi:X\oto Y)\to(\psi:X'\oto Y')$ in $\Arr(\QDist)$ to the continuous $\CQ$-distributor $\zeta:(X,\phi^*\phi_*)\oto(X',\psi^*\psi_*)$.

Since every Chu transform $(f,g):\phi\to\psi$ induces an arrow $(f_{\nat},g^{\nat}):\phi\to\psi$ in $\Arr(\QDist)$, there is a 2-functor
$$(\Box_{\nat},\Box^{\nat}):(\QChu)^{\co}\to\Arr(\QDist),\quad(f,g)\mapsto(f_{\nat},g^{\nat})$$
which is neutral on objects. As the commutative square
\begin{equation} \label{K-hK}
\bfig
\square/->`<-`<-`->/<1200,500>[\Arr(\QDist)`\QIntDist`(\QChu)^{\co}`(\QInt)^{\co};\hK`(\Box_{\nat},\Box^{\nat})`(-)_{\nat}`\sK^{\co}]
\efig
\end{equation}
reveals, $\hK$ may be viewed as an extension of the functor $\sK$. Moreover:

\begin{prop} \label{hK-full}
$\hK:\Arr(\QDist)\to\QIntDist$ is a full quantaloid homomorphism.
\end{prop}

\begin{proof}
It remains to show that $\hK$ is full. Given $\CQ$-distributors $\phi:X\oto Y$, $\psi:X'\oto Y'$, we need to show that
$$\hK:\Arr(\QDist)(\phi,\psi)\to\QIntDist((X,\phi^*\phi_*),(X',\psi^*\psi_*))$$
is surjective. To this end, for any continuous $\CQ$-distributor $\zeta:(X,\phi^*\phi_*)\oto(X',\psi^*\psi_*)$, we must find a $\CQ$-distributor $\eta:Y\oto Y'$ such that $(\zeta,\eta):\phi\to\psi$ is an arrow in $\Arr(\QDist)$. Indeed, let
$$\eta:=(\psi\circ\zeta)\lda\phi:Y\oto Y'.$$
Then
\begin{align*}
\eta(-,y')\circ\phi&=\phi^*\phi_*\zeta^*\psi(-,y')&(\eta=(\psi\circ\zeta)\lda\phi)\\
&=\phi^*\phi_*\zeta^*\psi^*\sy_{Y'}y'&(\text{Equations \eqref{tphi-def} and \eqref{tphi-Yoneda}})\\
&=\phi^*\phi_*\zeta^*\psi^*\psi_*\psi^*\sy_{Y'}y'&(\psi^*\dv\psi_*)\\
&=\zeta^*\psi^*\psi_*\psi^*\sy_{Y'}y'&(\text{Proposition \ref{cont-dist}\ref{cont-dist:l}})\\
&=\zeta^*\psi^*\sy_{Y'}y'&(\psi^*\dv\psi_*)\\
&=\psi(-,y')\circ\zeta&(\text{Equations \eqref{tphi-def} and \eqref{tphi-Yoneda}})
\end{align*}
for all $y'\in Y'$, as desired.
\end{proof}

Analogously to \eqref{K-hK}, there is a quantaloid homomorphism
$$\hI:\QIntDist\to\Arr(\QDist)$$
such that the square
\begin{equation} \label{I-hI}
\bfig
\square/<-`<-`<-`<-/<1200,500>[\Arr(\QDist)`\QIntDist`(\QChu)^{\co}`(\QInt)^{\co};\hI`(\Box_{\nat},\Box^{\nat})`(-)_{\nat}`\sI^{\co}]
\efig
\end{equation}
is commutative, and thus extends $\sI:\QInt\to\QChu$:

\begin{prop} \label{hI-functor}
For each continuous $\CQ$-distributor $\zeta:(X,\si)\oto(Y,\sj)$ between $\CQ$-interior spaces,
$$(\zeta,(\zeta^*)^{\nat}|_{\CO(X,\si),\CO(Y,\sj)}):(\ka_{\si}:X\oto\CO(X,\si))\to(\ka_{\sj}:Y\oto\CO(Y,\sj)))$$
is an arrow in $\Arr(\QDist)$, where $(\zeta^*)^{\nat}|_{\CO(X,\si),\CO(Y,\sj)}$ is the restriction of $(\zeta^*)^{\nat}:\PX\oto\PY$ on $\CO(X,\si)$ and $\CO(Y,\sj)$.
\end{prop}

\begin{proof}
Note that the $\CQ$-distributor $(\zeta^*)^{\nat}|_{\CO(X,\si),\CO(Y,\sj)}:\CO(X,\si)\oto\CO(Y,\sj)$ is well defined since $\zeta^*\lam\in\CO(X,\si)$ whenever $\lam\in\CO(Y,\sj)$ by Proposition \ref{cont-dist}\ref{cont-dist:l}. The conclusion then follows from
$$\ka_{\sj}(-,\lam)\circ\zeta=\lam\circ\zeta=\zeta^*\lam=\ka_{\si}(-,\zeta^*\lam)=1_{\CO(X,\si)}^{\nat}(-,\zeta^*\lam)\circ\ka_{\si}=(\zeta^*)^{\nat}|_{\CO(X,\si),\CO(Y,\sj)}(-,\lam)\circ\ka_{\si}$$
for all $\lam\in\CO(Y,b)$.
\end{proof}

The following proposition is an immediate consequence of Theorem \ref{I-dv-K} in combination with the definitions of $\hK$ and $\hI$:

\begin{prop} \label{hK-hI-id}
$\hK:\Arr(\QDist)\to\QIntDist$ is a left inverse of $\hI:\QIntDist\to\Arr(\QDist)$; hence, $\QIntDist$ is a retract of $\Arr(\QDist)$.
\end{prop}

\begin{rem}
As pointed out by the anonymous referee, it is worth considering the comma category
$$\QCat\da(-)^*$$
here, where $(-)^*:(\QDist)^{\op}\to\QCat$ sends each $\CQ$-distributor $\phi:X\oto Y$ to the $\CQ$-functor $\phi^*:\PY\to\PX$. This comma category may be identified with the category having $\CQ$-distributors as objects and pairs $(\zeta:X\oto X',\ g:Y'\to Y)$ consisting of a $\CQ$-distributor and a $\CQ$-functor satisfying
$$\psi\circ\zeta=g^{\nat}\circ\phi$$
as morphisms from $\phi:X\oto Y$ to $\psi:X'\oto Y'$. From Proposition \ref{hI-functor} it is easy to see that $\QIntDist$ can be embedded into $\QCat\da(-)^*$ by sending each continuous $\CQ$-distributor $\zeta:(X,\si)\oto(Y,\sj)$ to $(\zeta,\zeta^*|_{\CO(Y,\sj)}):\ka_{\si}\to\ka_{\sj}$, and analogously to Theorem \ref{I-dv-K} one may prove that $\QIntDist$ is a retract and a coreflective subcategory of ${\QCat\da(-)^*}$.
\end{rem}

\section{Diagonals between $\CQ$-distributors as continuous $\CQ$-distributors} \label{Q-IntDisto}

For continuous $\CQ$-distributors $\zeta,\zeta':(X,\si)\oto(Y,\sj)$ between $\CQ$-interior spaces, we denote by $\zeta\sim\zeta'$ if
\begin{equation} \label{cont-dist-equiv}
\zeta^*\sj=\zeta'^*\sj.
\end{equation}
To see the intuition of \eqref{cont-dist-equiv}, let us consider the case that $\zeta=f_{\nat}$ and $\zeta'=f'_{\nat}$ for some continuous $\CQ$-functors $f,f':(X,\si)\to(Y,\sj)$. Then \eqref{cont-dist-equiv} becomes
$$f^{\la}\sj=f'^{\la}\sj;$$
that is, $f\sim f'$ if the preimages of each open presheaf under $f$ and $f'$ are identical.

It is not difficult to see that ``$\sim$'' gives rise to a congruence on the quantaloid $\QIntDist$, and we denote the induced quotient quantaloid by
$$(\QIntDist)_{\so}:=\QIntDist/\!\sim.$$

\begin{prop} \label{zeta-equiv-K}
For arrows $(\zeta,\eta),(\zeta',\eta'):(\phi:X\oto Y)\to(\psi:X'\oto Y')$ in $\Arr(\QDist)$, the following statements are equivalent:
\begin{enumerate}[label={\rm(\roman*)}]
\item \label{zeta-equiv-K:diag} $(\zeta,\eta)\sim(\zeta',\eta'):\phi\to\psi$.
\item \label{zeta-equiv-K:int} $\zeta\sim\zeta':(X,\phi^*\phi_*)\oto(X',\psi^*\psi_*)$.
\end{enumerate}
\end{prop}

\begin{proof}
\ref{zeta-equiv-K:diag}$\implies$\ref{zeta-equiv-K:int}: If $\psi\circ\zeta=\psi\circ\zeta'$, then the functoriality of $(-)^*:(\QDist)^{\op}\to\QCat$ ensures that $\zeta^*\psi^*=\zeta'^*\psi^*$. Thus $\zeta^*\psi^*\psi_*=\zeta'^*\psi^*\psi_*$.

\ref{zeta-equiv-K:int}$\implies$\ref{zeta-equiv-K:diag}: If $\zeta^*\psi^*\psi_*=\zeta'^*\psi^*\psi_*$, then $$\zeta^*\psi^*=\zeta^*\psi^*\psi_*\psi^*=\zeta'^*\psi^*\psi_*\psi^*=\zeta'^*\psi^*.$$
It follows from \eqref{tphi-def} and \eqref{tphi-Yoneda} that
$$\psi(-,y')\circ\zeta=\zeta^*\psi^*\sy_{Y'}y'=\zeta'^*\psi^*\sy_{Y'}y'=\psi(-,y')\circ\zeta'$$
for all $y'\in Y'$. Thus $\psi\circ\zeta=\psi\circ\zeta'$.
\end{proof}

Proposition \ref{zeta-equiv-K} indicates that $\hK(\zeta,\eta)=\zeta:(X,\phi^*\phi_*)\oto(X',\psi^*\psi_*)$ is equal to $\hK(\zeta',\eta')=\zeta'$ in $(\QIntDist)_{\so}$ whenever $(\zeta,\eta)\sim(\zeta',\eta'):\phi\to\psi$ in $\Arr(\QDist)$. Hence, the universal property of the quotient quantaloid $\BD(\QDist)=\Arr(\QDist)/\!\sim$ ensures that there is a (unique) quantaloid homomorphism
$$\hKd:\BD(\QDist)\to(\QIntDist)_{\so}$$
making the square
$$\bfig
\square/->`->`-->`->/<1500,500>[\Arr(\QDist)`\BD(\QDist)`\QIntDist`(\QIntDist)_{\so};\sd`\hK`\hKd`\so]
\efig$$
commute, where $\sd$ and $\so$ are the obvious quotient homomorphisms (so that the composition of $\so$ and $\hK$ is also a full quantaloid homomorphism).

\begin{prop} \label{cont-dist-equiv-I}
For continuous $\CQ$-distributors $\zeta,\zeta':(X,\si)\oto(Y,\sj)$ between $\CQ$-interior spaces, the following statements are equivalent:
\begin{enumerate}[label={\rm(\roman*)}]
\item \label{cont-dist-equiv-I:int} $\zeta\sim\zeta':(X,\si)\oto(Y,\sj)$.
\item \label{cont-dist-equiv-I:diag} $(\zeta,(\zeta^*)^{\nat}|_{\CO(X,\si),\CO(Y,\sj)})=(\zeta',(\zeta'^*)^{\nat}|_{\CO(X,\si),\CO(Y,\sj)}):(\ka_{\si}:X\oto\CO(X,\si))\to(\ka_{\sj}:Y\oto\CO(Y,\sj)))$.
\end{enumerate}
\end{prop}

\begin{proof}
Just note that $(\zeta^*)^{\nat}|_{\CO(X,\si),\CO(Y,\sj)}=(\zeta'^*)^{\nat}|_{\CO(X,\si),\CO(Y,\sj)}$ means precisely $\zeta^*|_{\CO(Y,\sj)}=\zeta'^*|_{\CO(Y,\sj)}$, which is an alternative expression of \eqref{cont-dist-equiv}.
\end{proof}

Proposition \ref{cont-dist-equiv-I} shows that $\hI\zeta=\hI\zeta'$ whenever $\zeta\sim\zeta':(X,\si)\oto(Y,\sj)$. The universal property of the quotient quantaloid $(\QIntDist)_{\so}$ then guarantees the existence of a (unique) quantaloid homomorphism
$$\hIo:(\QIntDist)_{\so}\to\Arr(\QDist)$$
making the triangle
$$\bfig
\btriangle/<-`<--`->/<1500,500>[\Arr(\QDist)`\QIntDist`(\QIntDist)_{\so};\hI`\hIo`\so]
\efig$$
commute, and the composition of $\sd$ and $\hIo$ produces a quantaloid homomorphism
$$\hId:(\QIntDist)_{\so}\to\BD(\QDist).$$
$$\bfig
\square/->`@{->}@<-4pt>`@{-->}@<4pt>`->/<1500,500>[\Arr(\QDist)`\BD(\QDist)`\QIntDist`(\QIntDist)_{\so};\sd`\hK`\hKd`\so]
\btriangle|rra|/@{<-}@<4pt>`<--`/<1500,500>[\Arr(\QDist)`\QIntDist`(\QIntDist)_{\so};\hI`\hIo`]
\morphism(1500,0)|l|/@{-->}@<4pt>/<0,500>[(\QIntDist)_{\so}`\BD(\QDist);\hId]
\efig$$

From Proposition \ref{hK-hI-id} and the constructions of $\hKd$ and $\hId$ it is easy to conclude:

\begin{prop} \label{hKd-hId-id}
$\hKd:\BD(\QDist)\to(\QIntDist)_{\so}$ is a left inverse of $\hId:(\QIntDist)_{\so}\to\BD(\QDist)$.
\end{prop}

Note that Propositions \ref{hK-full} and \ref{zeta-equiv-K} guarantee that $\hKd$ is fully faithful, and Proposition \ref{hKd-hId-id} implies that $\hKd$ is surjective on objects. Therefore, we arrive at the main result of this paper:

\begin{thm} \label{hKd-hId-equiv}
$\hKd:\BD(\QDist)\to(\QIntDist)_{\so}$ and $\hId:(\QIntDist)_{\so}\to\BD(\QDist)$ establish an equivalence of quantaloids; hence, $\BD(\QDist)$ and $(\QIntDist)_{\so}$ are equivalent quantaloids.
\end{thm}

\begin{proof}
It remains to verify the claim about $\hId$. Indeed, since $\hKd$ is an equivalence of quantaloids, there exists a functor $\sF:(\QIntDist)_{\so}\to\BD(\QDist)$ such that $\sF\hKd$ is naturally isomorphic to the identity functor on $\BD(\QDist)$, thus so is $\hId\hKd$ as there are natural isomorphisms
$$\hId\hKd\cong\sF\hKd\hId\hKd\cong\sF\hKd,$$
showing that $\hId$ is also an equivalence of quantaloids.
\end{proof}

\begin{rem} \label{DQDist-dual-BQDist}
It has been elaborated in \cite[Subsection 1.1]{Shen2016a} that diagonals and back diagonals are dual constructions of each other. Their duality is once again supported by the equivalences of quantaloids
$$\BD(\QDist)\simeq(\QIntDist)_{\so}\quad\text{and}\quad\BB(\QDist)\simeq\QClsCloDist$$
given by \eqref{BQDist-QSup-QClsCloDist} and Theorem \ref{hKd-hId-equiv}, from the topological point of view:
\begin{itemize}
\item a diagonal between $\CQ$-distributors is essentially an equivalence class of continuous $\CQ$-distributors between $\CQ$-interior spaces;
\item a back diagonal between $\CQ$-distributors is essentially an equivalence class of continuous $\CQ$-distributors between $\CQ$-closure spaces.
\end{itemize}
\end{rem}

In the case that $\CQ={\bf 2}$ is the two-element Boolean algebra, $\BD(\Dist)$ is the Freyd completion of the quantaloid $\Dist$ of (pre)ordered sets and distributors, while $\IntDist$ is the quantaloid of ordered interior spaces (i.e., ordered sets $X$ equipped with an interior operator on its down-set lattice) and continuous distributors:

\begin{cor} \label{DDist-IntDisto-equiv}
The Freyd completion $\BD(\Dist)$ of the quantaloid $\Dist$ is equivalent to $\IntDist_{\so}$.
\end{cor}

\section{Diagonals between $\CQ$-relations as continuous $\CQ$-relations} \label{Q-IntRel}

Note that every set $X$ over $\ob\CQ$ is equipped with a \emph{discrete} $\CQ$-category structure, given by
$$\id_X(x,y)=\begin{cases}
1_{|x|} & \text{if}\ x=y,\\
\bot_{|x|,|y|} & \text{else}
\end{cases}$$
for all $x,y\in X$, where $\bot_{|x|,|y|}$ refers to the bottom element of the complete lattice $\CQ(|x|,|y|)$. A \emph{$\CQ$-relation}
$$\phi:X\rto Y$$
between sets over $\ob\CQ$ is precisely a $\CQ$-distributor
$$\phi:(X,\id_X)\oto(Y,\id_Y).$$
Hence, sets over $\ob\CQ$ and $\CQ$-relations constitute a full subquantaloid of $\QDist$, denoted by
$$\QRel.$$
It is easy to see that a $\CQ$-relation $\al:X\rto X$ defines a $\CQ$-category $(X,\al)$ if
\begin{equation} \label{rel-cat-def}
\id_X\leq\al\quad\text{and}\quad\al\circ\al\leq\al,
\end{equation}
and a $\CQ$-relation $\phi:X\rto Y$ becomes a $\CQ$-distributor $\phi:(X,\al)\oto(Y,\be)$ if
\begin{equation} \label{rel-dist-def}
\be\circ\phi\circ\al\leq\phi.
\end{equation}

\begin{prop} \label{DQDist-DQRel-equiv}
$\BD(\QDist)$ is equivalent to its full subquantaloid $\BD(\QRel)$.
\end{prop}

\begin{proof}
It suffices to show that every $\CQ$-distributor $\phi:(X,\al)\oto(Y,\be)$ is isomorphic to its underlying $\CQ$-relation $\phi:X\rto Y$ in the quantaloid $\BD(\QDist)$. Indeed, it is clear that the identity maps on $X$ and $Y$ are $\CQ$-functorial as
$$1_X:(X,\id_X)\to(X,\al)\quad\text{and}\quad 1_Y:(Y,\id_Y)\to(Y,\be).$$
It is routine to verify that
\begin{align*}
&((1_X)_{\nat},(1_Y)_{\nat}):(\phi:X\rto Y)\to(\phi:(X,\al)\oto(Y,\be))\quad\text{and}\\
&((1_X)^{\nat},(1_Y)^{\nat}):(\phi:(X,\al)\oto(Y,\be))\to(\phi:X\rto Y)
\end{align*}
$$\bfig
\square<1000,400>[(X,\id_X)`(X,\al)`(Y,\id_Y)`(Y,\be);(1_X)_{\nat}`\phi`\phi`(1_Y)_{\nat}]
\square(1800,0)<1000,400>[(X,\al)`(X,\id_X)`(Y,\be)`(Y,\id_Y);(1_X)^{\nat}`\phi`\phi`(1_Y)^{\nat}]
\place(500,0)[\circ] \place(500,400)[\circ] \place(0,200)[\circ] \place(1000,200)[\circ]
\place(2300,0)[\circ] \place(2300,400)[\circ] \place(1800,200)[\circ] \place(2800,200)[\circ]
\efig$$
are arrows in $\QDist$, and satisfy
\begin{align*}
&((1_X)^{\nat},(1_Y)^{\nat})\circ((1_X)_{\nat},(1_Y)_{\nat})=(\al,\be)\sim(\id_X,\id_Y):(\phi:X\rto Y)\to(\phi:X\rto Y),\\
&((1_X)_{\nat},(1_Y)_{\nat})\circ((1_X)^{\nat},(1_Y)^{\nat})=(\al,\be):(\phi:(X,\al)\oto(Y,\be))\to(\phi:(X,\al)\oto(Y,\be)),
\end{align*}
establishing an isomorphism between $\phi:(X,\al)\oto(Y,\be)$ and $\phi:X\rto Y$ in $\BD(\QDist)$.
\end{proof}

Similarly, we denote by
$$\QIntRel\quad\text{and}\quad(\QIntRel)_{\so}$$
the full subquantaloids of $\QIntDist$ and $(\QIntDist)_{\so}$, respectively, whose objects are restricted to $\CQ$-interior spaces $(X,\si)$ with $X$ being discrete.

\begin{prop} \label{QIntDisto-QIntRelo-equiv}
$(\QIntDist)_{\so}$ is equivalent to its full subquantaloid $(\QIntRel)_{\so}$.
\end{prop}

\begin{proof}
Suppose that $(X,\al)$ is a $\CQ$-category, i.e., $\al:X\rto X$ is a $\CQ$-relation satisfying \eqref{rel-cat-def}. If $\si:\sP(X,\al)\to\sP(X,\al)$ is a $\CQ$-interior operator, then
$$\si_0:\sP(X,\id_X)\to\sP(X,\id_X),\quad\si_0\mu:=\si(\mu\lda\al)$$
defines a $\CQ$-interior operator on $\sP(X,\id_X)$. Indeed, $\si_0\leq 1_{\sP(X,\id_X)}$ since
$$\si_0\mu=\si(\mu\lda\al)\leq\mu\lda\al\leq\mu\lda\id_X=\mu$$
for all $\mu\in\sP(X,\id_X)$. As for $\si_0=\si_0\si_0$, note that for any $\mu\in\sP(X,\id_X)$, $\si_0\mu\in\sP(X,\al)$ implies that $\si_0\mu=\si_0\mu\lda\al$, and $\si_0\mu\in\CO(X,\al,\si)$ implies that $\si\si_0\mu=\si_0\mu$. Thus
$$\si_0\mu=\si\si_0\mu=\si(\si_0\mu\lda\al)=\si_0\si_0\mu.$$

Now it suffices to show that $(X,\al,\si)$ is isomorphic to $(X,\id_X,\si_0)$ in the quantaloid $(\QIntDist)_{\so}$. Note that
\begin{equation} \label{OXa=OXa0}
\CO(X,\al,\si)=\CO(X,\id_X,\si_0).
\end{equation}
Indeed, on one hand, $\mu\in\CO(X,\al,\si)$ implies that $\si_0\mu=\si(\mu\lda\al)=\si\mu=\mu$, i.e., $\mu\in\CO(X,\id_X,\si_0)$. On the other hand, $\mu\in\CO(X,\id_X,\si_0)$ necessarily forces $\mu=\si_0\mu=\si(\mu\lda\al)\in\CO(X,\al,\si)$.

Since $1_X:(X,\id_X)\to(X,\al)$ is a $\CQ$-functor, its graph and cograph
$$(1_X)_{\nat}:(X,\id_X,\si_0)\oto(X,\al,\si),\quad (1_X)^{\nat}:(X,\al,\si)\oto(X,\id_X,\si_0)$$
are clearly continuous $\CQ$-distributors by \eqref{OXa=OXa0}, and satisfy
\begin{align*}
&(1_X)^{\nat}\circ(1_X)_{\nat}=\al\sim\id_X:(X,\id_X,\si_0)\oto(X,\id_X,\si_0),\\
&(1_X)_{\nat}\circ(1_X)^{\nat}=\al:(X,\al,\si)\oto(X,\al,\si),
\end{align*}
establishing an isomorphism between $(X,\al,\si)$ and $(X,\id_X,\si_0)$ in $(\QIntDist)_{\so}$.
\end{proof}

From Theorem \ref{hKd-hId-equiv} and Propositions \ref{DQDist-DQRel-equiv}, \ref{QIntDisto-QIntRelo-equiv} we soon deduce that:

\begin{thm} \label{DRel-QIntRelo-equiv}
$\BD(\QRel)$ and $(\QIntRel)_{\so}$ are equivalent quantaloids.
\end{thm}

In the case that $\CQ={\bf 2}$, $\BD(\Rel)$ is precisely the Freyd completion of the quantaloid $\Rel$ of sets and relations, while $\IntRel$ is the quantaloid of (classical) interior spaces (i.e., sets $X$ equipped with an interior operator on its powerset ${\bf 2}^X$) and continuous relations:

\begin{cor} \label{Rel-IntRelo-equiv}
The Freyd completion $\BD(\Rel)$ of the quantaloid $\Rel$ is equivalent to $\IntRel_{\so}$.
\end{cor}

Let $\BD(\Rel)_{\sf f}$ denote the full subquantaloid of $\BD(\Rel)$ whose objects are relations $\phi:X\rto Y$ such that $\tphi:Y\to{\bf 2}^X$ is injective and that
$$\sIm\tphi=\{\tphi y\mid y\in Y\}=\{\{x\in X\mid x\phi y\}\mid y\in Y\}$$
is closed under arbitrary union and finite intersection. Then, by Remark \ref{Pratt}, $\BD(\Rel)_{\sf f}$ is clearly equivalent to the full subquantaloid $\TopRel_{\so}$ of $\IntRel_{\so}$ consisting of topological spaces:

\begin{cor} \label{Relf-TopRelo-equiv}
$\BD(\Rel)_{\sf f}$ and $\TopRel_{\so}$ are equivalent quantaloids.
\end{cor}

\section{When $\CQ$ is a Girard quantaloid} \label{Q-Girard}

Given a quantaloid $\CQ$ and a family of $\CQ$-arrows $\FD=\{d_q:q\to q\}_{q\in\ob\CQ}$, we say that
\begin{itemize}
\item $\FD$ is a \emph{cyclic family}, if $d_p\lda u=u\rda d_q$ for all $\CQ$-arrows $u:p\lra q$;
\item $\FD$ is a \emph{dualizing family}, if $(d_p\lda u)\rda d_p=u=d_q\lda(u\rda d_q)$ for all $\CQ$-arrows $u:p\lra q$.
\end{itemize}
$\CQ$ is a \emph{Girard quantaloid} \cite{Rosenthal1992} if it is equipped with a cyclic dualizing family of $\CQ$-arrows. In this case, the \emph{complement} of a $\CQ$-arrow $u:p\to q$ is defined as
$$\neg u=d_p\lda u=u\rda d_q:q\lra p,$$
which clearly satisfies $\neg\neg u=u$, and it is straightforward to check that:

\begin{prop} \label{Girard-comp}
If $\CQ$ is a Girard quantaloid, then
$$v\circ u=\neg(\neg u\lda v)=\neg(u\rda\neg v)$$
for all $\CQ$-arrows $u:p\to q$, $v:q\to r$.
\end{prop}

The aim of this section is to show that, in the case that $\CQ$ is a small Girard quantaloid, we are able to concatenate the equivalences given by \eqref{BQDist-QSup-QClsCloDist} and Theorem \ref{hKd-hId-equiv}.

Recall that each quantaloid $\CQ$ induces a quantaloid $\ChuCon(\CQ)$, whose objects are $\CQ$-arrows and whose morphisms are \emph{Chu connections} \cite{Shen2016a} $(s,t):(u:p\to q)\to(v:p'\to q')$, i.e., pairs ($s:p\to p',\ t:q\to q'$) of $\CQ$-arrows satisfying
$$u\lda s=t\rda v.$$
$$\bfig
\square<1000,500>[p`p'`q`q';s`u`v`t]
\morphism(1000,500)|r|/-->/<-1000,-500>[p'`q;u\lda s=t\rda v]
\efig$$
For Chu connections $(s,t),(s',t'):u\to v$, we denote by $(s,t)\sim(s',t')$ if the squares
$$\bfig
\square<1000,500>[p`p'`q`q';s`u`v`t]
\square(1700,0)<1000,500>[p`p'`q`q';s'`u`v`t']
\morphism(1000,500)|r|/-->/<-1000,-500>[p'`q;u\lda s=t\rda v]
\morphism(2700,500)|r|/-->/<-1000,-500>[p'`q;u\lda s'=t'\rda v]
\efig$$
generate the same \emph{back diagonal}; that is, if
$$u\lda s=t\rda v=u\lda s'=t'\rda v.$$
``$\sim$'' gives rise to a congruence on $\ChuCon(Q)$, and the induced quotient quantaloid, denoted by
$$\BQ:=\ChuCon(\CQ)/\!\sim,$$
is called the quantaloid of \emph{back diagonals} \cite{Shen2016a} of $\CQ$.

\begin{prop} \label{Girard-Arr-ChuCon}
If $\CQ$ is a Girard quantaloid, then $\Arr(\CQ)$ and $\ChuCon(\CQ)$ are isomorphic quantaloids and, consequently, $\DQ$ and $\BQ$ are isomorphic quantaloids.
\end{prop}

\begin{proof}
Given $\CQ$-arrows $u:p\to q$, $v:p'\to q'$ and a pair ($s:p\to p',\ t:q\to q'$) of $\CQ$-arrows, it follows from Proposition \ref{Girard-comp} that
$$v\circ s=t\circ u\iff\neg u\lda t=s\rda\neg v;$$
$$\bfig
\square<1000,500>[p`p'`q`q';s`u`v`t]
\square(1700,0)<1000,500>[q`q'`p`p';t`\neg u`\neg v`s]
\morphism(0,500)|r|/-->/<1000,-500>[p`q';v\circ s=t\circ u]
\morphism(2700,500)|r|/-->/<-1000,-500>[q'`p;\neg u\lda t=s\rda\neg v]
\place(1350,250)[\iff]
\efig$$
that is, $(s,t):u\to v$ is an arrow in $\Arr(\CQ)$ if, and only if, $(t,s):\neg u\to\neg v$ is a Chu connection. Hence, the assignment $((s,t):u\to v)\mapsto((t,s):\neg u\to\neg v)$ defines an isomorphism of quantaloids
$$\neg:\Arr(\CQ)\to\ChuCon(\CQ)$$
which, clearly, also renders an isomorphism $\neg:\DQ\to\BQ$.
\end{proof}

\begin{rem}
The condition of $\CQ$ being Girard is indispensable for the isomorphism $\DQ\cong\BQ$. In the case that $\sQ$ is a commutative quantale, it is already known from \cite[Theorem 5.18]{Lai2020} that there is an isomorphism $\BD(\sQ)\cong\BB(\sQ)$ of quantaloids if, and only if, $\sQ$ is a \emph{Girard quantale} \cite{Rosenthal1990,Yetter1990}, i.e., a one-object Girard quantaloid.
\end{rem}

If $\CQ$ is a small Girard quantaloid and $X$ is a $\CQ$-category, then
$$(\neg 1_X^{\nat})(y,x)=\neg 1_X^{\nat}(x,y)$$
defines a $\CQ$-distributor $\neg 1_X^{\nat}:X\oto X$, and it is straightforward to check that
$$\{\neg 1_X^{\nat}:X\oto X\}_{X\in\ob(\QDist)}$$
is a cyclic dualizing family of $\QDist$. In fact:

\begin{prop} {\rm\cite{Rosenthal1992}} \label{QDist-Girard}
A small quantaloid $\CQ$ is Girard if, and only if, $\QDist$ is a Girard quantaloid.
\end{prop}

Therefore, Propositions \ref{Girard-Arr-ChuCon} and \ref{QDist-Girard} in conjunction with \eqref{BQDist-QSup-QClsCloDist} and Theorem \ref{hKd-hId-equiv} give rise to the following equivalences:

\begin{thm} \label{Girard-equiv}
If $\CQ$ is a small Girard quantaloid, then there are equivalences of quantaloids
$$\BD(\QDist)\simeq\BB(\QDist)\simeq(\QIntDist)_{\so}\simeq\QClsCloDist\simeq(\QSup)^{\op}.$$
\end{thm}

As a special case of Theorem \ref{Girard-equiv}, Corollary \ref{Rel-IntRelo-equiv} actually amounts to the following equivalences in the classical case:

\begin{cor} \label{Girard-equiv-classical}
There are equivalences of quantaloids
$$\BD(\Rel)\simeq\BB(\Rel)\simeq\IntRel_{\so}\simeq\ClsCloRel\simeq\Sup.$$
\end{cor}

\begin{rem}
The equivalences $\BB(\Rel)\simeq\Sup$ and $\ClsCloRel\simeq\Sup$ in Corollary \ref{Girard-equiv-classical} have appeared in \cite[Corollary 3.4.5]{Shen2016a} and \cite[Corollary 4.4.3]{Shen2016b}, respectively, where $\ClsCloRel$ is the quantaloid of (classical) closure spaces (i.e., sets $X$ equipped with a closure operator on its powerset ${\bf 2}^X$) and closed continuous relations, and the self-duality of the quantaloid $\Sup$ of complete lattices and join-preserving maps is applied here.
\end{rem}

\section*{Acknowledgement}

This work is supported by National Natural Science Foundation of China (No. 11701396). The author would like to thank Professor Hongliang Lai for helpful discussions. The author would also like to thank the anonymous referee for several helpful remarks.





\end{document}